\documentclass{article}
\usepackage{amsmath,amsfonts,amsthm,enumerate,array}
\usepackage[all]{xy}
\newtheorem{lemma}{Lemma}[section]
\newtheorem{definition}[lemma]{Definition}
\newtheorem{proposition}[lemma]{Proposition}

\newtheorem{theorem}[lemma]{Theorem}
\newtheorem{remark}[lemma]{Remark}
\newtheorem{condition}[lemma]{Condition}
\newtheorem{example}[lemma]{Example}
\newcommand{\cslat}{commutator semi-lattice}
\newcommand{\clat}{commutator lattice}
\newcommand{\jac}{Jacobi commutator semi-lattice}
\newcommand{\ass}{associative commutator lattice}
\newcommand{\asss}{associative commutator semi-lattice}
\newcommand{\C}{\mathbb{C}}
\newcommand{\I}{\mathbb{I}}
\newcommand{\sub}[1]{\mathbf{Sub}(#1)}
\newcommand{\nsub}[1]{\mathbf{NSub}(#1)}

\newcounter{tmp}
\begin{document}
\title{
Hall's criterion for nilpotence in semi-abelian categories
}
\author{J. R. A. Gray}
\maketitle
\begin{abstract}
A well-known theorem of P. Hall, usually called Hall's criterion for nilpotence, states: a group $G$ is nilpotent whenever it has a normal subgroup $N$ such that $G/[N,N]$ and $N$ are nilpotent. We widely generalize this result, replacing groups with objects in an abstract semi-abelian category satisfying suitable conditions. In particular, these conditions are satisfied in any algebraically coherent semi-abelian category, and hence in many categories of classical algebraic structures (including a few where Hall's theorem is already known). Note that the categories of crossed modules, crossed Lie algebras, $n$-cat groups, and cocommutative Hopf algebras over fields are also algebraically coherent. We give, however, a counter-example showing that Hall's theorem does not hold in the (semi-abelian) variety of non-associative rings.
\end{abstract}
\section{Introduction}
A well-known theorem of P.~Hall \cite{HALL:1958} (see also D.~J.~S.~Robinson \cite{ROBISON:1968}), called Hall's criterion for nilpotence, states: a group $G$ is nilpotent whenever it has a normal subgroup $N$ such that $G/[N,N]$ and $N$ are nilpotent.
Analogous results for Lie algebras and for certain types of non-associative algebras were proved by C.~Y.~Chao  \cite{CHAO:1968} and E.~L.~Stitzinger \cite{STITZINGER:1978}, respectively. Hall's theorem can be reformulated as a weak form of extension closedness for the class of nilpotent groups: if $p: E\to B$ is a surjective group homomorphism with nilpotent $B$ and the kernel contained in the commutator $[N,N]$ of a normal nilpotent subgroup $N$ of $E$, then $E$ is a nilpotent group.

This paper is devoted to a wide generalization of this fact (Theorem \ref{theorem: main}): we replace a surjective homomorphism with a regular epimorphism in an abstract semi-abelian category \cite{JANELIDZE_MARKI_THOLEN:2002} satisfying (i) Huq and Higgins commutators of normal monomorphisms coincide; (ii) Huq commutators distribute over binary joins (restricted to normal monomorphisms); (iii) the \emph{Jacobi identity} (Condition \ref{conditions} (b)(ii)) holds for Huq commutators.  These conditions turn out to be satisfied in any semi-abelian category that is algebraically coherent in the sense of \cite{CIGOLI_GRAY_VAN_DER_LINDEN:2015b}, and therefore in each of the above-mentioned concrete categories as well as in all categories of interest in the sense of G.~Orzech \cite{ORZECH:1972}. Note that algebraically coherent semi-abelian categories include the categories of: cocommutative Hopf algebras over an arbitrary field, (pre)crossed modules, crossed Lie algebras, and $n$-cat groups in the sense of J.-L.~Loday \cite{LODAY:1982} -- see Remark \ref{remark:main} for further details and references.

Our approach substantially uses what we call Jacobi commutator semi-lattices; it is a notion intermediate between commutator posets and associative commutator lattices in the sense of E. Mehdi-Nezhad \cite{MEHDI-NEZHAD:2015}.

\section{Jacobi commutator semi-lattices}
In this section we introduce a `theory of \jac{}s' and develop it essentially as far as necessary for the purposes of Section \ref{section: main}, where we use it in the proof of our main result on nilpotent objects in a semi-abelian category.

\begin{definition}\label{definition:cslat}(c.f. Definition 4.1 of \cite{MEHDI-NEZHAD:2015})
A \emph{\cslat{}} is a triple $(X,\leq,\cdot)$ where $X$ is a set, $\leq$ is a binary relation on $X$, and $\cdot$ is a binary operation on $X$ satisfying:
\begin{enumerate}[(a)]
\item $(X,\leq)$ is a join semi-lattice;
\item the operation $\cdot$ is commutative;
\item for each $a,b \in X$, $a\cdot b \leq b$;
\item for each $a,b,c \in X$, $a\cdot (b\vee c) = (a\cdot b) \vee  (a \cdot c)$.
\end{enumerate}
\begin{remark}
Note that for a \cslat{} $(X,\leq,\cdot)$ and an element $x$ in $X$, the map 
\[
\xymatrix@C=3ex@R=0.5ex{
X \ar[r]^{x\cdot-} & X\\
y\ar@{|->}[r] & x\cdot y
}
\]
is order preserving.
\end{remark}
\end{definition}
\begin{definition}\label{definition: jac}
A \cslat{} $(X,\leq,\cdot)$ is a \emph{\jac{}} if 
\begin{enumerate}[(a)]
\item \label{jacobi}for each $a,b,c \in X$, $a\cdot(b\cdot c) \leq ((a\cdot b) \cdot c) \vee (b\cdot (a\cdot c))$;
\end{enumerate}
\setcounter{tmp}{\arabic{enumi}}
and is an \emph{\asss{}} if
\begin{enumerate}[(a)]
\setcounter{enumi}{\arabic{tmp}}
\item \label{ass}for each $a,b,c \in X$, $a\cdot(b\cdot c) = (a\cdot b)\cdot c$.
\end{enumerate}
\end{definition}
\begin{definition}
A \emph{derivation} of a \cslat{} $(X,\leq,\cdot)$ is a map $f:X\to X$ which preserves joins and satisfies:
\begin{enumerate}[(a)]
\item for each $a,b\in X$, $f(a\cdot b) \leq (f(a)\cdot b) \vee (a\cdot f(b))$. 
\end{enumerate}
A derivation $f$ of a \cslat{} $(X,\leq,\cdot)$ is an \emph{inner derivation} if there exists $x$ in $X$ such that $f=x\cdot -$, that is, for each $a$ in $X$, $f(a)=x\cdot a$. 
\end{definition}
\begin{remark}\label{remark: simple}Note that for a \cslat{} $(X,\leq,\cdot)$:
\begin{enumerate}[(a)]
\item the Jacobi identity (Condition \ref{definition: jac} \eqref{jacobi}) easily follows from associativity (Condition \ref{definition: jac} \eqref{ass}), and is equivalent to the condition: for each $x$ in $X$, the map $x\cdot -$ is an inner derivation of $(X,\leq,\cdot)$
\item a map $f:X\to X$ which is bounded above by $1_X$ (that is $f(x)\leq x$ for each $x$ in $X$) and preserves joins is necessarily a derivation if for each $a,b \in X$, $f(a\cdot b) \leq f(a)\cdot f(b)$.
\end{enumerate}
\end{remark}
It is easy to see that every complete lattice admits a largest binary operation making it into a \clat{}. On the other hand it doesn't seem in general possible to construct a largest binary operation making it into a \jac{}. However, binary meet is the largest such operation when $(X,\leq)$ is a distributive lattice. In fact we have:
\begin{proposition}
If $(X,\leq)$ is a distributive lattice, then $(X,\leq, \wedge)$ is a \ass{} and a morphism $f: (X,\leq) \to (X,\leq)$ which preserves joins is a derivation if and only if it is bounded above by $1_X$.
\end{proposition}
\begin{proof}
It is immediate that $(X,\leq,\wedge)$ is a \ass{}. On the other hand if $f$ is a derivation, then for each $a$ in $X$, $f(a)=f(a\wedge a)\leq f(a)\wedge a\leq a$ and so $f$ is bounded above by $1_X$. The converse follows from Remark \ref{remark: simple} (b) since $f$, being an order preserving map, satisfies the condition: for each $a,b\in X$, $f(a\wedge b) \leq f(a)\wedge f(b)$.
\end{proof}
Let us note as an aside related to commutator theory of arithmetical categories that:
\begin{proposition}
For a \cslat{} $(X,\leq, \cdot)$ if $\cdot$ is idempotent, then $a\cdot b$ is the meet of $a$ and $b$, and $(X,\leq)$ is a distributive lattice.
\end{proposition}
\begin{proof}
For $w,a,b$ in $X$.
If $w \leq a$ and $w \leq b$, then $w = w\cdot w \leq a\cdot b$. Conversely if $w \leq a\cdot b$, then trivially $w \leq a$ and $w \leq b$. 
\end{proof}
The remainder of this section is devoted to showing that for $f$ a derivation bounded above by $1_X$ and for $g=x\cdot -$ an inner derivation of a \jac{} $(X,\leq,\cdot)$, if for some $y\geq x$ in $X$ and some positive integer $m$, $f^m(y)\leq g(x)$, then for each positive integer $k$ there exists a positive integer $m_k$ such that $f^{m_k}(y) \leq g^k(x)$. Our proof of this fact depends on a few facts.
 
By a routine argument using induction we obtain:
\begin{proposition}
Let $f$ be a derivation of a \cslat{} ${(X,\leq,\cdot)}$. For each $a, b$ in $X$ and for each non-negative integer $n$ we have
\[
f^n(a\cdot b) \leq \vee_{i=0}^n f^i(a)\cdot f^{n-i}(b).
\]
\end{proposition}
\begin{lemma}
Let $f$ be a derivation of a \jac{} $(X,\leq,\cdot)$ bounded above by $1_X$, let $x$ be an element of $X$, and let $g : X\to X$ be the map defined by $g(s)=x\cdot s$.
If for some positive integer $m$, $f^{m}(x) \leq g(x)$, then for each positive integer $k$, 
\begin{align}\label{equation:1}
f^{n_k}(g^{k-1}(x)) \leq g^{k}(x)\end{align}
where $n_k = k(m-1)+1$.
\end{lemma}
\begin{proof}
We prove the claim by induction on $k$. Since $n_1=m$ the claim holds by assumption when $k$ is $1$. Suppose \eqref{equation:1} holds for some positive integer $k$. By the previous proposition we have
\[
f^{n_{k+1}}(g^k(x))=f^{n_{k+1}}(x\cdot g^{k-1}(x)) \leq \vee_{i=0}^{n_{k+1}} f^i(x)\cdot f^{n_{k+1}-i}(g^{k-1}(x)).
\]
However, for $i \leq m-1$ since  $n_{k+1}-i\geq n_{k}$ we have \[
f^{n_{k+1}-i}(g^{k-1}(x))\leq f^{n_k}(g^{k-1}(x)) \leq g^k(x)\] and hence $f^{i}(x)\cdot f^{n_{k+1}-i}(g^{k-1}(x)) \leq x\cdot g^{k}(x)=g^{k+1}(x)$. On other hand if $i \geq m$, then 
$f^{i}(x) \leq g(x)=x\cdot x$ and so \begin{align*}
f^{i}(x) \cdot f^{n_{k+1}-i}(g^{k-1}(x)) &\leq (x\cdot x) \cdot g^{k-1}(x)\\ &= g^{k-1}(x)\cdot (x\cdot x) \\&\leq ((g^{k-1}(x)\cdot x)\cdot x)\vee (x\cdot (g^{k-1}(x)\cdot x))\\ &= g^{k+1}(x).\end{align*} 
Therefore, $f^{n_{k+1}}(g^{k}(x)) \leq g^{k+1}(x)$ and the claim holds by induction.
\end{proof}
\begin{proposition}\label{proposition:jac main}
Let $f$ be a derivation of a \jac{} $(X,\leq,\cdot)$ bounded above by $1_X$, let $x$ and $y$ be elements of $X$, and let $g$ be the inner derivation of $(X,\leq,\cdot)$ defined for each $s$ in $X$ by $g(s)=x\cdot s$. If $x\leq y$ and for some positive integer $m$, $f^m(y) \leq g(x)$, then for each positive integer $k$,
\begin{align}\label{equation:2}
 f^{m_k}(y) \leq g^k(x)
\end{align}
where $m_k = \frac{k(k+1)}{2}(m-1)+k$.
\end{proposition}
\begin{proof}
We prove the claim by induction on $k$. Since $m_1=m$ the claim holds by assumption when $k=1$. 
Suppose \eqref{equation:2} holds for some positive integer $k$.
Since $f^m(x)\leq f^m(y) \leq g(x)$ and $m_{k+1} = n_{k+1} + m_{k}$ where $n_{k+1}$ is as in the previous lemma it follows by the previous lemma that
\[
f^{m_{k+1}}(y) = f^{n_{k+1}}(f^{m_{k}}(y)) \leq f^{n_{k+1}}(g^{k}(x)) \leq g^{k+1}(x).
\]
This completes the induction step and hence the proof of the claim.
\end{proof}
\section{The main result}\label{section: main}
In this section we prove our main result which as mentioned uses Proposition \ref{proposition:jac main} for its proof.

Let us recall the necessary background to explain precisely in which context we will prove our generalization.
Semi-abelian categories were introduced, by G.~Janelidze, L.~Marki, W.~Tholen in \cite{JANELIDZE_MARKI_THOLEN:2002}, to play a similar role for the categories of groups and algebras, as abelian categories play for the categories of abelian groups and modules. A category $\C$ is semi-abelian if it is pointed, Barr-exact \cite{BARR:1971}, Bourn protomodular \cite{BOURN:1991} and has finite coproducts. In a semi-abelian category there is a natural notion of when a pair of morphisms with common codomain commute. Note that this notion was first considered by S.~Huq \cite{HUQ:1968,HUQ:1971} in a closely related context.
 To recall how this notion is defined let us introduce some notation. We denote by $0$ a zero object in $\C$, that is, an object which is both terminal and initial. We also denote by $0$ each zero morphism, that is, a morphism which factors through a zero object. For $A$ and $B$ in $\C$ we denote by $(A\times B,\pi_1,\pi_2)$ a product of $A$ and $B$ in $\C$, and write $\langle 1,0\rangle: A\to A\times B$ and $\langle 0,1\rangle : B \to A\times B$ for the unique morphisms with $\pi_1\langle 1,0\rangle = 1_A$, $\pi_2\langle 1,0\rangle =0$, $\pi_1 \langle 0,1\rangle=0$ and $\pi_2 \langle 0,1\rangle= 1_B$.
A pair of morphisms $f:A\to C$ and $g: B\to C$ in $\C$, commute, if there is a (necessarily unique) morphism $\varphi: A\times B \to C$ making the diagram 
\begin{equation}\label{diag:commutes}
\xymatrix{
A \ar[r]^-{\langle 1,0\rangle}\ar@/_2ex/[dr]_{f} & A\times B \ar[d]^{\varphi} & B\ar[l]_-{\langle 0, 1\rangle}\ar@/^2ex/[dl]^{g}\\
& C &
}
\end{equation} 
commute.
More generally the Huq commutator of $f:A\to C$ and $g:B\to C$ is defined to be the smallest normal subobject $N$ of $C$ such that $qf$ and $qg$ commute, where $q:C\to C/N$ is the cokernel of the associated normal monomorphism $N\to C$. It turns out that Huq commutators always exist in semi-abelian categories and, as shown by D.~Bourn in \cite{BOURN:2004}, the \emph{commutator quotient} $q: C\to C/N$ can be constructed as part of the colimiting cone of the outer arrows of \eqref{diag:commutes}.  For subobjects $S$ and $T$ of an object $C$ in $\C$ we will write $[S,T]_C$ for the Huq commutator of the associated monomorphisms $S \to C$ and $T\to C$. We write $\sub{C}$ and $\nsub{C}$ for the lattices of subobjects and normal subobjects of $C$, respectively. Recall also that $C$ is nilpotent \cite{HUQ:1968} if there exists a non-negative integer $n$ such that $\gamma_C^n(C)=0$, where $\gamma_C$ is the map sending $S$ in $\sub{C}$ to $[C,S]_C$ in $\sub{C}$. The least such $n$ is the nilpotency class of $C$. 

We will prove our generalization in the context of a semi-abelian category $\C$ satisfying the following conditions:
\begin{condition}\label{conditions}\ 
\begin{enumerate}[(a)]
\item for $K,L \leq S\leq C$, if $K$ and $L$ are normal in $C$, then $[K,L]_S=[K,L]_C$ (in $\sub{C}$);
\item for $K,L,M$ normal in $C$:
\begin{enumerate}[(i)]
\item  $[K,L\vee M]_C= [K,L]_C \vee [K,M]_C$ in $\sub{C}$;
\item $[K,[L,M]_C]_C\leq [[K,L]_C,M]_C \vee [L,[K,M]_C]_C$ in $\sub{C}$.
\end{enumerate}
\end{enumerate}
\end{condition}
Before we state and prove our generalization, let us briefly recall how each of these conditions fits into the literature and make a few other comments:
\begin{remark}\label{remark:main}\ 
\begin{itemize}
\item  According to Theorem 2.8 of \cite{CIGOLI_GRAY_VAN_DER_LINDEN:2015a}, Condition \ref{conditions} (a) is equivalent to the coincidence of the Huq and Higgins commutators of normal subobjects;
\item If $\C$ is a semi-abelian algebraically cartesian closed category (in the sense of \cite{BOURN_GRAY:2012}, first considered in \cite{GRAY:2010a,GRAY:2012b}), then $\C$ has centralizers and hence Condition \ref{conditions} (b) (i) holds (see \cite{GRAY:2013a});
\item If $\C$ is semi-abelian and the Huq and Smith-Pedicchio commutators coincide (in the sense of \cite{MARTINS-FERREIRA_VAN_DER_LINDEN:2012}), then Condition \ref{conditions} (b) (i) again holds, since the Smith-Pedicchio commutator distributes over joins in an exact Mal'tsev category with coequalizers \cite{PEDICCHIO:1995}. 
\item We call Condition \ref{conditions} (b) (ii), the Jacobi identity. It is a special case of the so-called \emph{three subgroups lemma} in the context of the category of groups. Note that according to P.~Hall in \cite{HALL:1958} in the context of groups what we are calling the Jacobi identity was first proved by himself in \cite{HALL:1933}, while the above mentioned three subgroups lemma was proved later by L.~Kaluzhnin in \cite{KALUZHNIN:1950};
\item Condition \ref{conditions} follows in a semi-abelian category from algebraic coherence in the sense of \cite{CIGOLI_GRAY_VAN_DER_LINDEN:2015b} (see Theorems 6.18, 6.27 and 7.1 there);
\item Examples of algebraically coherent semi-abelian categories include the categories of groups, rings, Lie algebras over a commutative ring, and all categories of interest in the sense of G.~Orzech \cite{ORZECH:1972}. The category of cocommutative Hopf algebras over an arbitrary field has been shown by M.~Gran, F.~Sterck and J.~Vercruysse to be algebraically coherent and semi-abelian \cite{GRAN_STERCK_VERCRUYSSE:2019}. Further examples can be obtained from the fact that if $\C$ is an algebraically coherent semi-abelian category, then so is: each functor category $\C^\I$ for an arbitrary category $\I$; each full subcategory of $\C$ closed under products and subobjects; each category of split epimorphism in $\C$ with codomain $B$ where $B$ is an object in $\C$ (see Corollary 3.5 and Propositions 3.6 and 3.7 of \cite{CIGOLI_GRAY_VAN_DER_LINDEN:2015b}). In particular, it follows that the categories of (pre)crossed modules (of groups), crossed Lie algebras, and $n$-cat groups (in the sense of J.-L.~Loday \cite{LODAY:1982}) are all examples, as well as, each category of internal (pre)crossed modules (in the sense of G.~Janelidze in \cite{JANELIDZE:2003}) of an algebraically coherent semi-abelian category.
\end{itemize}
\end{remark}

Recalling that in a semi-abelian category the binary join of normal subobjects as subobjects is necessarily normal, we see that Condition \ref{conditions} implies:
\begin{proposition}\label{proposition: main}
If $\C$ is a semi-abelian category satisfying Condition \ref{conditions}, then for each object $C$ in $\C$ the triple $(\nsub{C},\leq,\cdot)$ where $(\nsub{C},\leq)$ is the lattice of normal subobjects of $C$, and $\cdot$ is defined by $K\cdot L= [K,L]_C$ is a \jac{}.  
\end{proposition}
As mentioned in Remark \ref{remark:main}, Theorems 6.18, 6.27 and 7.1 of \cite{CIGOLI_GRAY_VAN_DER_LINDEN:2015b} imply that every algebraically coherent semi-abelian category satisfies Condition \ref{conditions}. We have:
\begin{theorem}\label{theorem: main}
Let $\C$ be an algebraically coherent semi-abelian category (more generally a semi-abelian category satisfying Condition \ref{conditions})
and let $p:E \to B$ be an extension of a nilpotent object $B$ in $\C$.
If the kernel of $p$ is contained in the Huq commutator $[N,N]_N$ of a nilpotent normal subobject $N$ of $E$, then $E$ is nilpotent. Furthermore, if $N$ is of nilpotency class $c$ and $B$ is of nilpotency class $d$, then $E$ is of nilpotency class at most $\frac{c(c+1)}{2}(d-1) +c$.
\end{theorem}
\begin{proof}
Let $X=\nsub{E}$ and let $f : X\to X$ and $g:X\to X $ be the maps defined by $f(K)=[E,K]_E$ and $g(K)=[N,K]_E$.
Since Huq commutators are preserved by regular image along regular epimorphisms (see e.g.\ Theorem 5.2 of \cite{GRAN_JANELIDZE_URSINI:2014}) it follows that if $B$ is nilpotent of nilpotency class $d$, then $p(f^d(E))=0$ and so $f^d(E)$ is a subobject of the kernel of $p$ and hence a subobject of $g(N)$. Since $(X,\leq,\cdot)$, with the corresponding interpretation as in Proposition \ref{proposition: main}, is a \jac{}, and hence $f$ is a derivation bounded above by $1_X$, it follows from Proposition \ref{proposition:jac main} that for any positive integer $c$, \[f^{\frac{c(c+1)}{2} (d-1) +c}(E) \leq g^c(N).\] The claim is completed by noting that by Condition \ref{conditions} (a) for $K \leq N$, $g(K) = [N,K]_N$ and hence if $N$ is nilpotent of nilpotency class $c$, then $g^c(N)=0$.
\end{proof}
Let us end by showing that the previous theorem does not hold in the category of non-associative rings (i.e.\ not necessarily associative distributive rings). As explained at the end of \cite{GRAY:2013a}, the category of non-associative rings is a semi-abelian category satisfying Condition \ref{conditions} (b) (i) which is not algebraically cartesian closed (note the misprint there: $i_3$ should be $\langle 0,0,1,0\rangle$).
\begin{example} 
Let $B$ and $E$ be the non-associative rings with underlying abelian groups the free abelian groups on $\{b_1,b_2\}$ and $\{e_1,e_2,e_3\}$, respectively, and with multiplication defined on generators by
\[
\begin{tabular}{>{$}l<{$}|*{2}{>{$}l<{$}}}
 \cdot & b_1 & b_2 \\
\hline\vrule height 12pt width 0pt
b_1 & 0 & 0\\
b_2 & 0 & 0\\
\end{tabular}\hspace{2cm}
\begin{tabular}{>{$}l<{$}|*{3}{>{$}l<{$}}}
 \cdot & e_1 & e_2 & e_3\\
\hline\vrule height 12pt width 0pt
e_1 & 0 & 0 & 0\\
e_2 & 0 & e_3 & 0\\
e_3 & e_3 & 0 & 0.\\
\end{tabular}
\]
Note that this makes $B$ abelian and hence nilpotent.
An easy calculation shows that unique abelian group homomorphism $p:E\to B$ with $p(e_1)=b_1$, $p(e_2)=b_2$ and $p(e_3)=0$ is a non-associative ring homomorphism. We will show that $p$ satisfies the remaining requirements of the theorem but $E$ is not nilpotent. It is easy to see that the kernel of $p$ is $X=\langle\{e_3\}\rangle$ (i.e.\ the subgroup of $E$ generated by $e_3$). Now let $N=\langle\{e_2,e_3\}\rangle$ and note that it is a subring of $E$. Recall that if $S$ and $T$ are subrings of a non-associative ring $R$, then the Huq commutator $[S,T]_R$ is the smallest ideal of $R$ containing $\{s\cdot t\,|\,s \in S, t\in T\}\cup\{t\cdot s\,|\, s\in S, t\in T\}$.
Since $e_2^2=e_3$ and $e_3e_1=e_3$ it follows that $X$ is contained in both $[N,N]_N$ and $[E,X]_X$. On the other hand since $\{e_i\cdot e_j| i,j\in \{1,2,3\}\}=\{0, e_3\}$ it follows that $[N,N]_N=[E,E]_E=[E,X]_E=X$. Since $e_3^2=e_3\cdot e_2= e_2\cdot e_3=0$ it follows that $[N,X]_E=0$. Therefore $[N,[N,N]_N]_N=[N,X]_N=0$ and $[E,[E,E]_E]_E=[E,X]_E=X=[E,E]_E$ and so $N$ nilpotent but $E$ is not.
\end{example}
\providecommand{\bysame}{\leavevmode\hbox to3em{\hrulefill}\thinspace}
\providecommand{\MR}{\relax\ifhmode\unskip\space\fi MR }
\providecommand{\MRhref}[2]{%
  \href{http://www.ams.org/mathscinet-getitem?mr=#1}{#2}
}
\providecommand{\href}[2]{#2}

\end{document}